\documentclass{amsart}
\usepackage{amssymb,latexsym}
\usepackage{amscd,amsthm}
\usepackage{mathtools}
\usepackage{mathrsfs}
\usepackage{leftidx}
\usepackage{marginnote}
\usepackage[usenames,dvipsnames,svgnames,table]{xcolor}
\usepackage{inputenc}

\usepackage[all]{xy}
\usepackage{tikz}

\usepackage{tikz-cd}

\usepackage[linkbordercolor=BrickRed, citebordercolor=OliveGreen]{hyperref}

\newtheorem{theorem}{Theorem}[section]

\newtheorem*{introthm}{Theorem}

\newtheorem{lemma}[theorem]{Lemma}

\newtheorem{proposition}[theorem]{Proposition}
\newtheorem{corollary}[theorem]{Corollary}

\theoremstyle{definition}
\newtheorem{definition}[theorem]{Definition}
\newtheorem{fact}[theorem]{Fact}
\newtheorem{remark}[theorem]{Remark}

\newtheorem{ipg}[theorem]{}

\DeclareMathOperator{\Ext}{Ext}
\DeclareMathOperator{\Hom}{Hom}

\DeclareMathOperator{\Ch}{\mathbf{C}}

\DeclareMathOperator{\Proj}{Proj}
\DeclareMathOperator{\Inj}{Inj}
\DeclareMathOperator{\ModR}{Mod-\mathnormal{R}}

\DeclareMathOperator{\Rmod}{\mathnormal{R}-mod}
\DeclareMathOperator{\Flat}{Flat}
\DeclareMathOperator{\PProj}{PProj}

\DeclareMathOperator{\FPI}{FPI}
\DeclareMathOperator{\Filt}{Filt}

\newcommand{\class}[1]{\mathcal{#1}}   %% font for classes

\newcommand{\fp}{\textnormal{fp}}

\begin{document}
\emergencystretch 3em

\title{A note on homotopy categories of FP-Injectives}

\author{Georgios Dalezios}
\address{Departamento de Matem\'aticas, Universidad de Murcia, 30100 Murcia, Spain}
\email{georgios.dalezios@um.es}
\address{Department of Mathematical Sciences, University of Copenhagen, Universitets\-parken 5, 2100 Copenhagen {\O}, Denmark}

\thanks{The author is supported by the Fundaci\'{o}n S\'{e}neca of Murcia 19880/GERM/15.}
\date{}


\begin{abstract}
For a locally finitely presented Grothendieck category $\class A$, we consider a certain subcategory of the homotopy category of FP-injectives in $\class A$ which we show is compactly generated. In the case where $\class A$ is locally coherent, we identify this subcategory with the derived category of FP-injectives in $\class A$. 
Our results are, in a sense, dual to the ones obtained by Neeman \cite{ANm08} on the homotopy category of flat modules. Our proof is based on extending a characterization of the pure acyclic complexes which is due to Emmanouil \cite{Emm-pure-acyclic}. 
\end{abstract}

\maketitle

\section{Introduction}
The work of Neeman \cite{ANm08} on the homotopy category of flat modules has led to interesting advances in ring theory and homological algebra. Neeman was inspired by work of Iyengar and Krause \cite{SInHKr06}, who proved that over a Noetherian ring $R$ with a dualizing complex $D$, the composite $\mathbf{K}(\Proj R)\rightarrowtail\mathbf{K}(\Flat R)\xrightarrow{D\otimes_{R}-}\mathbf{K}(\Inj R)$ is an equivalence, which restricts to Grothendieck duality. Neeman in \cite{ANm08} focuses on the embedding $\mathbf{K}(\Proj R)\rightarrowtail\mathbf{K}(\Flat R)$ for a general ring $R$ and shows that the category $\mathbf{K}(\Proj R)$ is $\aleph_{1}$--compactly generated
%that the embedding $i$ admits a right adjoint whose kernel consists of the exact complexes of flat modules which have flat cycles $\mathbf{K}_{\mathrm{pac}}(\Flat R)$, 
and that the composite of canonical maps $\mathbf{K}(\Proj R)\rightarrowtail{}\mathbf{K}(\Flat R)\xrightarrow{}\mathbf{D}(\Flat R)$ is an equivalence.
%{\color{red}In subsequent work \cite{Neeman-some-adjoints} Neeman proves the existence of a \textit{recollement}
%\begin{equation*}
%  \xymatrix@C=3pc{
% \mathbf{K_{\mathrm{pac}}}(\Flat R)    \ar@<0.0ex>[r] |-{\mathrm{inc}}   &  
%     \mathbf{K}(\Flat R) \ar@<0.9ex>[l]    \ar@<-0.9ex>[l]  \ar@<0.0ex>[r]&
%   \mathbf{D}(\Flat R)  \ar@<-0.9ex>[l] |-{\mathrm{inc}}  \ar@<0.9ex>[l] 
%  }\!.
%\end{equation*}
%} 
Related work on homotopy categories and the existence of adjoints between them is done by Krause \cite{Krause-approx-adj}, Murfet and Salarian \cite{DMfSSl11}, Saor\'in and \v{S}{\fontencoding{T1}\selectfont \v{t}}ov\'i\v{c}ek \cite{Saorin-Stoviceck-exact}, and others.
%Moreover, Emmanouil \cite{Emm-pure-acyclic} makes a self-contained study related to the category $\mathbf{K_{\mathrm{pac}}}(\Flat R)$.

Closely related to the notion of flatness, is the notion of purity~\cite{Cohn-purity}. A submodule $A\leq B$ is called pure if any finite system of linear equations with constants from $A$ and a solution in $B$, has a solution in $A$. This condition can be expressed diagrammatically, and is equivalent to asking for the sequence $A\rightarrowtail B\twoheadrightarrow B/A$ to remain exact after applying, for any finitely presented module $F$, the functor $\Hom_{R}(F,-)$, or equivalently the functor $F\otimes_{R}-$. Such sequences are called pure exact and they are of interest since they form the smallest class of short exact sequences which is closed under filtered colimits. It follows from this discussion that a module $M$ is flat if and only if any epimorphism with target $M$ is pure. Thus flatness can be defined in any additive category which has an appropriate notion of finitely presented objects, namely locally finitely presented additive categories \cite{locally-fp-Breitsprecher,WCB94}. If $\class A$ is such a category, it is well known that the relation between purity and flatness can be given formally via the equivalence $\class A\cong\Flat(\fp(\class A)^{\mathrm{op}},\mathrm{Ab});\,\, A\mapsto\Hom_{\class A}(-,A)|_{\fp(\class A)}$\footnote{\, $\fp(\class A)$ denotes a set of isomorphism classes of finitely presented objects in $\class A$, see \ref{Locally fp}.}, see~\cite[1.4]{WCB94}. Thus, roughly speaking, the study of purity can be reduced to the study of flat (left exact) functors, and Neeman's results have analogues in the context of purity, see Emmanouil \cite{Emm-pure-acyclic}, Krause \cite{Krause-approx-adj}, Simson \cite{Simson-pure-acyclic} and \v{S}{\fontencoding{T1}\selectfont \v{t}}ov\'i\v{c}ek \cite{Stoviceck-on-purity}.
%{\color{blue}Using this technique, combined with nontrivial results, Krause \cite{Krause-approx-adj} exhibits a recollement concerning the homotopy category of pure projectives  (these are the objects which are projective with respect to the pure exact sequences).}

The dual notion of flatness, in a locally finitely presented Grothendieck category $\class A$, is that of FP-injectivity. Namely, an object $A$ in $\class A$ is called FP-injective if any monomorphism with source $A$ is pure. We denote the class of FP-injective objects by $\FPI(\class A)$. FP-injective modules were studied first by Stenstr\"om in \cite{Stenstrom-FPI}. One reason why they are of importance is because over (non-Noetherian) rings where injectives fail to be closed under coproducts, one can work with FP-injectives which are always closed under coproducts. Moreover, a ring is coherent if and only if the class of FP-injective modules is closed under filtered colimits, in strong analogy with the dual situation, where coherent rings are characterized by the closure of flat modules under products.

In this note our goal is to provide, in a sense, duals of the above mentioned results of Neeman, that is, to obtain analogous results for the homotopy category of FP-injectives. For this we look at the tensor embedding functor of a module category to FP-injective (right exact) functors, that is, the functor
%In this note, we are motivated by the tensor To understand what corresponds to the homotopy category $\mathbf{K}(\Proj R)$ in the dual picture, we look at the tensor embedding functor 
$\ModR\rightarrow\class A:=(\Rmod,\mathrm{Ab}) ; \,\, M\mapsto (M\otimes_{R}-)|_{\Rmod}$, which identifies pure exact sequences in $\ModR$ with short exact sequences of FP-injective (right exact) functors, and induces an equivalence $\ModR\cong\FPI(\class A)$ \cite[\S1]{LGrCUJ81}. It is easy to observe that under this equivalence, the pure projective modules (the projectives with respect to the pure exact sequences) correspond to functors in the class $\FPI(\class A)\cap\leftidx{^{\bot}}\FPI(\class A)$\footnote{\,\, $\leftidx{^{\bot}}\FPI(\class A)$ denotes the left orthogonal class of FP-injectives with respect to the $\Ext^{1}_{\class A}(-,-)$ functor, see~\ref{Cotorsion pairs}.}. We point out that by work of Eklof and Trlifaj \cite{PCEJTr01}, we know that this class consists of those FP-injectives which are (summands of) transfinite extensions of finitely presented objects (see \ref{facts for our favorite cotorsion pair}). We can now state our main result (\ref{main thm}).

%As with the case of flat (left exact) functors, the study of purity can be reduced to the study of FP-injective (right exact) functors \cite{LGrCUJ73}. More precisely, if $R$ is an associative ring, 

\begin{introthm}Let $\class A$ be a locally finitely presented Grothendieck category and denote by $\FPI(\class A)$ the class of FP-injective objects in $\class A$. Then the homotopy category $\mathbf{K}(\FPI(\class A)\cap\leftidx{^{\bot}}\FPI(\class A))$ is compactly generated. Moreover, if $\class A$ is locally coherent, the composite functor
\[\mathbf{K}(\FPI(\class A)\cap\leftidx{^{\bot}}\FPI(\class A))\rightarrowtail\mathbf{K}(\FPI(\class A))\xrightarrow{\mathrm{can}}\mathbf{D}(\FPI(\class A))\]
is an equivalence of triangulated categories.
\end{introthm}

Our proof is based on Neeman's strategy. Let $\class C:=\FPI(\class A)\cap\leftidx{^{\bot}}\FPI(\class A)$. Since $\mathbf{K}(\class C)$ is compactly generated and admits coproducts, we obtain a right adjoint of the inclusion $\mathbf{K}(\class C)\rightarrowtail\mathbf{K}(\FPI(\class A))$, and in the case where $\class A$ is locally coherent, we identify its kernel with the pure acyclic complexes of FP-injective objects in $\class A$. From this it follows that any chain map from a complex in $\mathbf{K}(\class C)$ to a pure acyclic complex of FP-injectives is null homotopic. In fact, in \ref{Emm +} we prove something more general, namely that for any locally finitely presented Grothendieck category $\class A$, any chain map from a complex in $\mathbf{K}(\leftidx{^{\bot}}\FPI(\class A))$ to a pure acyclic complex is null homotopic. This extends a result of Emmanouil \cite{Emm-pure-acyclic}.
%The compact generation of $\mathbf{K}(\FPI(\class A)\cap\leftidx{^{\bot}}\FPI(\class A))$ is proved in \ref{Thm on compact generation}. For the rest of the statement, we give two proofs. The first (\ref{Model category}) involves model theoretic techniques while the second (\ref{Brown}) uses Brown representability. 

Finally, we point out that \v{S}{\fontencoding{T1}\selectfont \v{t}}ov\'i\v{c}ek \cite{Stoviceck-on-purity} has also studied the category $\mathbf{D}(\FPI(\class A))$ and in the locally coherent case has proved the existence of a model category structure with homotopy category $\mathbf{D}(\FPI(\class A))$. Using our main result we can identify the cofibrant objects in this model structure with the category $\mathbf{K}(\class C)$ (see \ref{Model category}, \ref{Stov}).

%Our theorem combined with one of the main results in \cite{Krause-approx-adj} gives the following (\ref{recollement}).

%\begin{introcor}
%Let $\class A$ be a locally coherent Grothendieck category. Then there exists a recollement
%\begin{equation*}
%  \xymatrix@C=3pc{
% \mathbf{K_{\mathrm{pac}}}(\FPI(\class A))    \ar@<0.0ex>[r] |-{\mathrm{inc}}   &  
%     \mathbf{K}(\FPI(\class A)) \ar@<0.9ex>[l]    \ar@<-0.9ex>[l]  \ar@<0.0ex>[r]&
%   \mathbf{D}(\FPI(\class A))  \ar@<-0.9ex>[l] |-{\mathrm{inc}}  \ar@<0.9ex>[l] 
%  }\!.
%\end{equation*}
%\end{introcor}
%In the last section we sketch how the tensor embedding can be used in order to obtain the analogous recollement (in the module case) concerning pure projective modules from \cite{Krause-approx-adj}.
\section{Preliminaries} 
\begin{ipg}\textbf{Locally finitely presented additive categories.} \cite{locally-fp-Breitsprecher,WCB94}
\label{Locally fp}
In an additive category $\class A$, an object $X$ is called \textit{finitely presented} if the functor $\Hom_{\class A}(X,-):\class A\rightarrow\mathrm{Ab}$ preserves filtered colimits. $\class A$ is called \textit{locally finitely presented} if it is cocomplete, the isomorphism classes of finitely presented objects in $\class A$ form a set $\mathrm{fp}(\class A)$, and every object in $\class A$ is a filtered colimit of objects in $\mathrm{fp}(\class A)$. An abelian category $\class A$ is locally finitely presented if and only if it is a Grothendieck category with a generating set of finitely presented objects \cite[Satz 1.5]{locally-fp-Breitsprecher}. 
A locally finitely presented Grothendieck category $\class A$ is called \textit{locally coherent} if the subcategory $\fp(\class A)$ is abelian.
\end{ipg}

\begin{ipg}\textbf{Purity.}
\label{Purity}
If $\class A$ is a locally finitely presented additive category, a sequence $0\rightarrow X\rightarrow Y\rightarrow Z\rightarrow 0$ in $\class A$ is called \textit{pure exact} if it is $\Hom_{\class A}(\fp(\class A),-)$--exact, that is, if for any $A\in\fp(\class A)$, the sequence \[0\rightarrow\Hom_{\class A}(A,X)\rightarrow\Hom_{\class A}(A,Y)\rightarrow\Hom_{\class A}(A,Z)\rightarrow 0\] is an exact sequence of abelian groups. An object $X\in\class A$ is called \textit{pure projective} if any pure exact sequence of the form $0\rightarrow Z\rightarrow Y\rightarrow X\rightarrow 0$ splits, and dually $X$ is called \textit{pure injective} if any pure exact sequence of the form $0\rightarrow X\rightarrow Y\rightarrow Z\rightarrow 0$ splits. We will denote the class of pure projective objects in $\class A$ by $\PProj(\class A)$. Pure exact sequences induce on the category $\class A$ the structure of a (Quillen) exact category, i.e. we equip $\class A$ with an exact structure \cite[Dfn.~2.1]{Buhler-exact-cat} where the conflations are the pure exact sequences; see Crawley-Boevey~\cite[3.1]{WCB94}. We will refer to this exact structure as the pure exact structure on $\class A$.
\end{ipg}

\begin{ipg}\textbf{Cotorsion Pairs.} (\cite{Salce-cotorsion},~see~also~\cite{GobelTrlifaj})
\label{Cotorsion pairs}
Let $\class X$ be a class of objects in an exact category $\class A$. Put
\[\class X\leftidx{^{\bot}}:=\{A\in\class A\, |\, \forall X\in\class X, \,\, \Ext^{1}_{\class A}(X,A)=0\}\]
and define $^{\bot} \class X$ analogously. A pair $(\class X,\class Y)$ of classes in $\class A$ is called a \textit{cotorsion pair} if $\class X ^{\bot}=\class Y$ and $^{\bot}\class Y=\class X$. 
A cotorsion pair is said to be \textit{generated by a set}\footnote{This terminology is in accordance with G\"obel and Trlifaj \cite[Dfn.~2.2.1]{GobelTrlifaj}.} if it is of the form $(\leftidx{^{\bot}}(\mathcal{S}\leftidx{^{\bot}}),\mathcal{S}\leftidx{^{\bot}})$ where $\mathcal{S}$ is a set of objects in $\class A$. A cotorsion pair $(\class X,\class Y)$ is called \textit{complete} if for every object $A$ in $\class A$ there exists a short exact sequence 
$0\rightarrow Y\rightarrow X\rightarrow A\rightarrow 0$ with $X\in\class X$ and $Y\in\class Y$, and also a short exact sequence 
$0\rightarrow A\rightarrow Y\rightarrow X\rightarrow 0$ with $X\in\class X$  and $Y\in\class Y$. It is called \textit{hereditary} if $\class X$ is closed under kernels of epimorphisms and $\class Y$ is closed under cokernels of monomorphisms.
\end{ipg}
We recall a fundamental result on cotorsion pairs generated by a set from the work of Eklof and Trlifaj \cite{PCEJTr01}. First a definition.

\begin{definition}
Let $\class A$ be an abelian category and $\mathcal{S}$ a class of objects in $\class A$. An object $A$ in $\class A$ is called \textit{$\class S$-filtered} if there exists a chain of subobjects 
\[0=A_{0}\subseteq A_{1}\subseteq ...\subseteq \bigcup\limits_{\alpha<\sigma}A_{\alpha}=A\]
where $\sigma$ is an ordinal, $A_{\lambda}=\cup_{\beta<\lambda}A_{\beta}$ for all limit ordinals $\lambda$, and $A_{\alpha+1}/A_{\alpha}\in\class S$ for all $\alpha<\sigma$. The class of $\class S$-filtered objects will be denoted by $\Filt(\class S)$.
\end{definition}

\begin{fact}(\cite{PCEJTr01}, see also \cite[3.2]{GobelTrlifaj})
\label{facts for our favorite cotorsion pair}
Let $\class S$ be a (small) set of objects in a Grothendieck category and assume that $\class S$ contains a generator. Then the following hold:
\begin{itemize}
\item[(i)] The cotorsion pair $(\leftidx{^{\bot}}(\class S\leftidx{^{\bot}}),\class S\leftidx{^{\bot}})$ is complete.
\item[(ii)] The class $\leftidx{^{\bot}}(\class S\leftidx{^{\bot}})$ consists of direct summands of $\class S$-filtered objects, that is, for all $X\in \leftidx{^{\bot}} (\class S\leftidx{^{\bot}} )$ we have $P\cong X\oplus K$ where $P\in\Filt(S)$. Moreover, in this decomposition $K$ can be chosen in $\leftidx{^{\bot}}(\class S\leftidx{^{\bot}})\cap \class S\leftidx{^{\bot}}.$
\end{itemize}
\end{fact}

\begin{ipg}\textbf{FP-Injectives.}~\cite{Stenstrom-FPI} Let $\class A$ be a locally finitely presented Grothendieck category. The objects in the class $\fp(\class A)\leftidx{^{\bot}}=\{A\in\class A\, |\, \forall F\in\fp(\class A),\,\, \Ext^{1}_{\class A}(F,A)=0\}$ are called \textit{FP-injective objects}. We will denote this class by $\FPI(\class A)$.
 
Note that fact \ref{facts for our favorite cotorsion pair}, applied on the cotorsion pair $(\leftidx{^{\bot}}\FPI(\class A),\FPI(\class A))$, tells us that the class $\leftidx{^{\bot}}\FPI(\class A)$ consists of direct summands of $\fp(\class A)$-filtered modules. 
 \end{ipg}
 
We recall the following well known facts for the class of FP-injectives.
\begin{fact} (\cite{Stenstrom-FPI},~see~also~\cite[App.B]{Stoviceck-on-purity})
\label{Facts for FPI}
Let $\class A$ be a locally finitely presented Grothendieck category. Then the following hold:
\begin{itemize}
\item[(i)] The class $\FPI(\class A)$ is closed under extensions, direct unions, products, coproducts, and pure subobjects.
\item[(ii)] An object $A\in\class A$ belongs to $\FPI(\class A)$ if and only if any monomorphism with source $A$ is pure.
\end{itemize}
Moreover, the category $\class A$ is locally coherent if and only if the class $\FPI(\class A)$ is closed under filtered colimits if and only if the class $\FPI(\class A)$ is closed under cokernels of monomorphisms.
%In this case, it is easy to see that the cotorsion pair $(\leftidx{^{\bot}}\FPI(\class A),\FPI(\class A))$ is hereditary.
\end{fact}

\begin{ipg}\textbf{The derived category of an exact category.} \cite{Neeman-derived-cat-of-exact-cat} 
\label{The derived cat of an exact cat}
Let $\class E$ be a (Quillen) exact category and denote by $\Ch(\class E)$ the corresponding category of chain complexes. $\Ch(\class E)$ has a canonical exact structure in which a diagram $X \rightarrowtail Y \twoheadrightarrow Z$ in $\Ch(\class E)$ is a conflation if and only if $X_n \rightarrowtail Y_n \twoheadrightarrow Z_n$ is a conflation in $\class E$ for every $n \in \mathbb{Z}$; see B\"uhler \cite[Lem.~9.1]{Buhler-exact-cat}. We will refer to this exact structure as the \textit{induced exact structure} on $\Ch(\class E)$.

A complex $X=\cdots\rightarrow X_{n+1}\xrightarrow{d_{n+1}} X_{n}\xrightarrow{d_{n}} X_{n-1}\rightarrow\cdots$ in $\Ch(\class E)$ is called \textit{acyclic} (with respect to the exact structure of $\class E$) if each map $d_{n}$ decomposes in $\class E$ as a deflation $X_{n}\twoheadrightarrow Z_{n-1}(X)$ followed by an inflation $Z_{n-1}(X)\rightarrowtail X_{n-1}$ and such that the induced sequence $Z_{n}(X)\rightarrowtail X_{n}\twoheadrightarrow Z_{n-1}(X)$ is a conflation in $\class E$. Denote by $\mathbf{K}_{\mathrm{ac}}(\class E)$ the homotopy category of acyclic complexes. If the exact category $\class E$ has split idempotents, then $\mathbf{K}_{\mathrm{ac}}(\class E)$ is an \'epaisse (thick) subcategory of $\mathbf{K}(\class E)$ \cite[1.2]{Neeman-derived-cat-of-exact-cat} and then by definition \cite[1.5]{Neeman-derived-cat-of-exact-cat} the \textit{derived category} of $\class E$ is the Verdier quotient $\mathbf{D}(\class E):=\mathbf{K}(\class E)/\mathbf{K}_{\mathrm{ac}}(\class E)$.
\end{ipg}

\section{On the homotopy category of fp-injectives}

Let $\class A$ is a locally finitely presented Grothendieck category, viewed as an exact category with the pure exact structure, as in \ref{Purity}. Then the acyclic complexes in $\class A$ (with respect to this exact structure) are called \textit{pure acyclic complexes} and are denoted by $\Ch_{\mathrm{pac}}(\class A)$. Moreover, the subcategory of FP-injective objects in $\class A$ is closed under extensions, therefore it is an exact category. It is easy to see that the acyclic complexes in $\Ch(\FPI(\class A))$ (with respect to the exact category $\FPI(\class A))$ are the acyclic complexes (in the usual sense) with cycles in $\FPI(\class A)$. Equivalently, since the class $\FPI(\class A)$ is closed under pure subobjects (\ref{Facts for FPI}), they are the pure acyclic complexes with components FP-injectives. Thus we will denote them by $\Ch_{\mathrm{pac}}(\FPI(\class A))$. 
Then by definition (\ref{The derived cat of an exact cat}) we have $\mathbf{D}(\FPI(\class A)):=\mathbf{K}(\FPI(\class A))/\mathbf{K}_{\mathrm{pac}}(\FPI(\class A))$. 

In Theorem \ref{main thm}, we identify $\mathbf{D}(\FPI(\class A))$ with a certain homotopy category. The key ingredient is to extend a characterization of the pure acyclic complexes which is due to Emmanouil.
In \cite{Emm-pure-acyclic} Emmanouil proves that a complex $X$ is pure acyclic if and only if any chain map from a complex of pure projectives to $X$ is null homotopic. Emmanouil's proof is self-contained, while Simson \cite{Simson-pure-acyclic} and also \v{S}{\fontencoding{T1}\selectfont \v{t}}ov\'i\v{c}ek \cite{Stoviceck-on-purity} give a functorial proof of this result by reducing it to Neeman's \cite[Thm. 8.6]{ANm08}.

We first recall a useful and well known lemma we will need.

\begin{lemma}
\label{lemma}
Let $\class A$ be an exact category and consider $\Ch(\class A)$ with the induced exact structure (as in \ref{The derived cat of an exact cat}). If $X,Y\in\Ch(\class A)$, denote by $\Ext^{1}_{\Ch(\class A)}(X,Y)$ the abelian group of (Yoneda) extensions with respect to the induced exact structure, and by $\Ext^{1}_{\mathrm{dw}(\class A)}(X,Y)$ the subgroup consisting of extensions $Y\rightarrowtail T\twoheadrightarrow X$ which are degreewise split. Then we have natural isomorphisms
\[\Ext^{1}_{\mathrm{dw(\class A)}}(X,\Sigma^{-(n+1)}Y)\cong\mathrm{H}_{n}\Hom_{\class A}(X,Y)\cong\Hom_{\mathbf{K}(\class A)}(X,\Sigma^{-n}Y).\]
\end{lemma}

\begin{proposition} \textnormal{(Compare with \cite{Emm-pure-acyclic})}
\label{Emm +}
Let $\class A$ be a locally finitely presented Grothendieck category and let $X$ be a chain complex in $\class A$. Then the following are equivalent:
\begin{itemize}
\item[(i)] $X$ is a pure acyclic complex.
\item[(ii)] Any chain map from a complex in $\Ch(\PProj(\class A))$ to $X$ is null-homotopic.
\item[(iii)] Any chain map from a complex in $\Ch(\leftidx{^{\bot}}\FPI(\class A))$ to $X$ is null-homotopic.
\end{itemize}
In particular, any pure acyclic complex with components in $\leftidx{^{\bot}}\FPI(\class A)$ is contractible.
\end{proposition}

\begin{proof}
As we discussed above the assertions $(i)\Leftrightarrow(ii)$ have been proved in \cite{Emm-pure-acyclic}. Moreover, $(iii)\Rightarrow(ii)$ is trivial, thus we are left with $(ii)\Rightarrow(iii)$. 
First consider the case where we are given a chain map $Y\rightarrow X$, with $Y$ having components in $\Filt(\fp(\class A))$. From the fact that each component of $Y$ is $\fp(\class A)$-filtered, a result of \v{S}{\fontencoding{T1}\selectfont \v{t}}ov\'i\v{c}ek \cite[Prop.~4.3]{Stoviceck-Hill-lemma} implies that $Y$ itself is $\Ch^{-}(\fp(\class A))$-filtered, that is, $Y$ is given as a continuous chain of subcomplexes
\[0=Y_{0}\subseteq Y_{1}\subseteq ...\subseteq \bigcup\limits_{\alpha<\sigma}Y_{\alpha}=Y\]
where $\sigma$ is an ordinal, $Y_{\lambda}=\cup_{\beta<\lambda}Y_{\beta}$ for all limit ordinals $\lambda$, and for all $\alpha<\sigma$ the quotient $Y_{\alpha+1}/Y_{\alpha}$ is a bounded below complex with components finitely presented objects. Now, denote by $\Ch(\class A)_{\mathrm{pure}}$ the exact category of chain complexes with the induced pure exact structure (as in \ref{The derived cat of an exact cat}). For all ordinals $\alpha<\sigma$ we have 

\begin{equation}
\begin{split}
\Ext^{1}_{\Ch(\class A)_{\mathrm{pure}}}(Y_{\alpha+1}/Y_{\alpha},X) & =  \Ext_{\mathrm{dw}(\class A)}^{1}(Y_{\alpha+1}/Y_{\alpha},X) \\
 & \cong \Hom_{\mathbf{K}(\class A)}(Y_{\alpha+1}/Y_{\alpha},\Sigma^{1}X) \\
 & = 0 \nonumber
 \end{split}
\end{equation}
where the first equality holds because each degreewise pure extension of a complex with pure projective components is degreewise split exact, the isomorphism is obtained by Lemma \ref{lemma} and the
last equality follows by assumption. Hence Eklof's lemma \cite[Lemma 1]{PCEJTr01}, in its version for exact categories \cite[Prop. 2.12]{Saorin-Stoviceck-exact}, gives the result.
Now consider the case where $Y$ has components in $\leftidx{^{\bot}}\FPI(\class A)$. Then from \ref{facts for our favorite cotorsion pair} we know that for all $n\in\mathbb{Z}$ there exists $J_{n}$ such that $Y_{n}\oplus J_{n}\cong F_{n}$, where $F_{n}$ is $\fp(\class A)$-filtered. Consider for each $n$ the disc complex $D_{n}(J_{n})=0\rightarrow J_{n}=J_{n}\rightarrow 0$, which is concentrated in homological degrees $n$ and $n-1$. Then the complex
\[Y':=Y\oplus\left(\bigoplus\limits_{n\in\mathbb{Z}}D_{n}(J_{n})\right)\oplus \Sigma^{-1}Y\]
has components of the form $F_{n}\oplus F_{n+1}$, and these are $\fp(\class A)$-filtered. Then from the previous treated case we have that $\Hom_{\mathbf{K}(\class A)}(Y',X)=0$, thus $\Hom_{\mathbf{K}(\class A)}(Y,X)=0$ too. Finally, by what we have proved, if $X$ is a pure acyclic complex with components in the class $\leftidx{^{\bot}}\FPI(\class A)$, then the identity map on $X$ is null homotopic, in other words $X$ is contractible.
\end{proof}

As a corollary we obtain a result on pure periodicity which extends the following fact: if $M$ is a module fitting into a pure exact sequence $0\rightarrow M\rightarrow P\rightarrow M\rightarrow 0$ with $P$ pure projective, then $M$ is pure projective as well. In other words, every $\PProj(\class A)$--pure periodic module is pure projective. This result was first proved by Simson \cite[Thm.~1.3]{Simson-pacific} and recently by Emmanouil \cite[Cor.~3.6]{Emm-pure-acyclic}. We point out that in \cite{Estrada-Bazzoni-Izurdiaga} the authors provide a proof of this result and also a proof of the dual statement.

Our version below extends the case of $\PProj(\class A)$--pure periodicity to the case of $\leftidx{^{\bot}}\FPI(\class A)$--pure periodicity.

\begin{corollary}
\textnormal{(Compare with \cite{Emm-pure-acyclic,Simson-pacific})}
Let $\class A$ be a locally finitely presented Grothendieck category and let $M$ be an object in $\class A$ admitting a pure short exact sequence of the form \[0\rightarrow M\rightarrow F\rightarrow M\rightarrow 0\]
%where  $F_{i}\in\leftidx{^{\bot}}\FPI(\class A)$, for all $i=0,1,...,n$. Then $M\in\leftidx{^{\bot}}\FPI(\class A)$.
where $F\in\leftidx{^{\bot}}\FPI(\class A)$. Then $M\in\leftidx{^{\bot}}\FPI(\class A)$.
In other words, any $\leftidx{^{\bot}}\FPI(\class A)$--pure periodic object belongs to the class $\leftidx{^{\bot}}\FPI(\class A)$.
\end{corollary}

\begin{proof}
The argument is identical as in \cite[Cor. 3.8]{Emm-pure-acyclic}, but invoking \ref{Emm +}. Namely, we may splice copies of the given short exact sequence to obtain a pure acyclic complex with components in $\leftidx{^{\bot}}\FPI(\class A)$, thus a contractible complex. Hence $M$ is a summand of $F\in\leftidx{^{\bot}}\FPI(\class A)$ and the assertion follows since the class $\leftidx{^{\bot}}\FPI(\class A)$ is closed under summands. 
\end{proof}

We now relate Proposition \ref{Emm +} with the theory of cotorsion pairs. 

\begin{lemma}
\label{lemma-pairs}
Let $\class A$ be a locally finitely presented Grothendieck category and let $\class C:=\Ch(\FPI(\class A)\cap\leftidx{^{\bot}}\FPI(\class A))$ and\, $\class W:=\Ch_{\mathrm{pac}}(\FPI(\class A))$. Then the following hold.
\begin{itemize}
\item[(i)] $(\Ch(\leftidx{^{\bot}}\FPI(\class A)),\class W)$ is a cotorsion pair in $\Ch(\class A)$.
\item[(ii)] If $\class A$ is locally coherent, then the cotorsion pair of (i) is complete. Moreover, in this case the pair $(\class C,\class W)$ is a complete and hereditary cotorsion pair in $\Ch(\FPI(\class A))$.
\end{itemize}
\end{lemma}

\begin{proof}
(i) Recall that by definition $(\leftidx{^{\bot}}\FPI(\class A),\FPI(\class A))$ is a cotorsion pair which is generated by a set, therefore by \ref{facts for our favorite cotorsion pair} it is complete. Thus, from work of Gillespie \cite[Prop.~3.6]{Gillespie-flat-model-2004}, there exists an induced cotorsion pair $(^{\bot}\class W,\class W)$ in the abelian category $\Ch(\class A)$ where the class $^{\bot}\class W$ can be identified with 
\[^{\bot}\class W=\{X\in\Ch(\leftidx{^{\bot}}\FPI(\class A))\, |\, \forall W\in\class W,\,\, \Hom_{\mathbf{K}(\class A)}(X,W)=0\}.\]
Since every complex in $\class W$ is pure acyclic, \ref{Emm +} implies that $^{\bot}\class W=\Ch(\leftidx{^{\bot}}\FPI(\class A))$, which proves the claim.

(ii) Assume that $\class A$ is locally coherent. In this case, from \ref{Facts for FPI}, we obtain that the complete cotorsion pair $(\leftidx{^{\bot}}\FPI(\class A),\FPI(\class A))$ is also hereditary. Thus, \cite[Cor.~3.7]{Gillespie-flat-sheaves} implies that the cotorsion pair of (i) is complete.

Now, we prove that $(\class C,\class W)$ is a cotorsion pair in $\Ch(\FPI(\class A))$. Let $C\in\class C$ and $W\in\class W$. Invoking Lemma \ref{lemma} we have
\begin{equation}
\Ext^{1}_{\Ch(\FPI(\class A))}(C,W) 
= \Ext_{\mathrm{dw}(\class A)}^{1}(C,W) 
\cong \Hom_{\mathbf{K}(\class A)}(C,\Sigma^{1}W). \nonumber
\end{equation}
Since $W$ is pure acyclic, from part (i) we obtain $^{\bot}\mathcal{W}=\class C$ and $\class W\subseteq\class C^{\bot}$.
%To prove the inclusion $^{\bot}\class W\subseteq\class C$, let $Y\in\Ch(\FPI(\class A))$ be such that, $\forall W\in\class W;\, \Ext^{1}_{\Ch(\FPI(\class A))}(Y,W)=0$. We need to show that $Y$ has components in $\leftidx{^{\bot}}\FPI(\class A)$. Let $M\in\FPI(\class A)$, then for each $n\in\mathbb{Z}$ we have a natural isomorphism 
%\[\Ext^{1}_{\class A}(Y_{n},M)\cong\Ext^{1}_{\Ch(\class A)}(Y,D_{n+1}(M))\] 
%where $D_{n+1}(M)$ denotes the disc complex of $M$, which is concentrated in homological degrees $n+1$ and $n$. By assumption, the abelian group on the right hand side is zero, which proves the desired inclusion.

To prove the inclusion $\class C^{\bot}\subseteq\class W$, let $X\in\Ch(\FPI(\class A))$ be such that, for all $C\in\class C$,\, $\Ext^{1}_{\Ch(\FPI(\class A))}(C,X)=0$. We need to show that $X\in\class W$. Since the cotorsion pair $(\Ch(\leftidx{^{\bot}}\FPI(\class A)),\class W)$ is complete, there exists a short exact sequence 
$X\rightarrowtail W\twoheadrightarrow C$
with $W\in\class W$ and $C\in\Ch(\leftidx{^{\bot}}\FPI(\class A))$. Since $\class A$ is locally coherent, from \ref{Facts for FPI} we obtain that the complex $C$ has components in $\FPI(\class A)$. By the assumption on $X$ this short exact sequence splits, therefore the fact that $\class W$ is closed under direct summands implies that $X\in\class W$.

Completeness of the cotorsion pair $(\class C,\class W)$ in $\Ch(\FPI(\class A))$ follows easily from the completeness of the cotorsion pair in (i). We show that $(\class C,\class W)$ is hereditary. The category $\class C$, as a subcategory of $\Ch(\FPI(\class A))$, is easily seen to be closed under kernels of epimorphisms. To see that the class $\class W$ is closed under cokernels of monomorphisms let $0\rightarrow A\rightarrow B\rightarrow C\rightarrow 0$ be a short exact sequence in $\Ch(\FPI(\class A))$ with $A, B\in\class W$. Then $C$ is an exact complex and using the fact that in the coherent case $\FPI(\class A)$ is closed under cokernels of monomorphisms (\ref{Facts for FPI}), we obtain that $C$ has cycles in $\FPI(\class A)$, thus $C\in\class W$.
\end{proof}

Recall that if $\class T$ is a triangulated category with set-indexed coproducts, an object $S\in\class T$ is called \textit{compact} if for any family $\{X_{i}\}_{i\in I}$ of objects in $\class T$, the natural map $\coprod_{i\in I} \Hom_{\class T}(X_{i},S)\rightarrow\Hom_{\class T}(\coprod_{i\in I}X_{i},S)$ is an isomorphism. $\class T$ is called \textit{compactly generated} if there exists a set $\class S$ of compact objects in $\class T$, such that for any non-zero $T\in\class T$ there exists a non-zero morphism $S\rightarrow T$ for some $S\in\class S$.

\begin{theorem}
\label{main thm}
Let $\class A$ be a locally finitely presented Grothendieck category. Then the homotopy category $\mathbf{K}(\FPI(\class A)\cap\vphantom{a}\leftidx{^{\bot}}\FPI(\class A))$ is compactly generated.
Moreover, if $\class A$ is locally coherent, the composite functor
\[\mathbf{K}(\FPI(\class A)\cap\leftidx{^{\bot}}\FPI(\class A))\rightarrowtail\mathbf{K}(\FPI(\class A))\xrightarrow{\mathrm{can}}\mathbf{D}(\FPI(\class A))\]
is an equivalence of triangulated categories.
\end{theorem}

\begin{proof}
Put $\class C:=\FPI(\class A)\cap\leftidx{^{\bot}}\FPI(\class A)$. We will make use of \cite[Thm.~3.1]{HHlPJr07b}, which asserts that for any class of objects $\class C$ which is closed under (set indexed) coproducts and direct summands, the homotopy category $\mathbf{K}(\class C)$ is compactly generated, provided the following hold:
\begin{itemize}
\item[(i)] Every finitely presented object $A$ has a \emph{right $\class C$-resolution} \cite[Dfn.~8.1.2]{rha}, which by definition means that there exists a sequence $0\rightarrow A\rightarrow C_{0}\rightarrow C_{1}\rightarrow\cdots$
with $C_{i}\in\class C$ which is exact after applying functors of the form $\Hom_{\class A}(-,\class C)$.
\item[(ii)] Every pure exact sequence consisting of objects in $\class C$ is split exact.
\end{itemize}
Their result holds for modules over associative rings, but it is clear that it generalizes to our setup. To check condition (i), recall that the cotorsion pair $(\leftidx{^{\bot}}\FPI(\class A),\FPI(\class A))$ is complete, therefore for any $A\in\fp(\class A)$, we can construct an exact sequence \[C(A):=\,\,\,\,\,\,\,\,\,\,\,\,\,\,\,0\rightarrow A\xrightarrow{\partial^{-1}} C_{0}\xrightarrow{\partial^{0}} C_{1}\xrightarrow{\partial^{1}} C_{2}\rightarrow\cdots \]
where $\partial^{-1}$ is a (special) FP-injective preenvelope of $A$ with cokernel $\epsilon_{0}:C_{0}\twoheadrightarrow Z_{0}\in\leftidx{^{\bot}}\FPI(\class A)$, $\partial^{0}=d_{0}\circ\epsilon_{0}$ where $d_{0}$ is an FP-injective envelope of $Z_{0}$ with cokernel $C_{1}\twoheadrightarrow Z_{1}\in\leftidx{^{\bot}}\FPI(\class A)$ etc. Since $\fp(\class A)$ is contained in $\leftidx{^{\bot}}\FPI(\class A)$ and the latter class is closed under extensions, we deduce that for all $i=0,1,...$;\, $C_{i}\in\leftidx{^{\bot}}\FPI(\class A)$. The sequence constructed has all the $C_{i}$'s in $\class C$ and clearly is $\Hom_{\class A}(-,\class C)$--exact, thus it is a right $\class C$--resolution of $A$.

To check condition (ii), let $C:=\,\,\cdots\rightarrow C_{n+1}\rightarrow C_{n}\rightarrow C_{n-1}\rightarrow\cdots$ be a pure exact sequence consisting of objects in $\class C$. In particular, $C$ is a pure acyclic complex with components in $\leftidx{^{\bot}}\FPI(\class A)$, hence by \ref{Emm +} it is contractible. Thus employing \cite[Thm.~3.1]{HHlPJr07b} we obtain that $\mathbf{K}(\class C)$ is compactly generated by the set
$\{\Sigma^{i}C(A)\, |\, A\in\fp(\class A),\,\, i\in\mathbb{Z}\}$.

We now assume that $\class A$ is locally coherent. Since $\mathbf{K}(\class C)$ is compactly generated and the inclusion $j_{!}:\mathbf{K}(\class C)\rightarrow\mathbf{K}(\FPI(\class A))$ preserves coproducts (which exist because $\FPI(\class A)$ is closed under coproducts), by Neeman's Brown representability theorem \cite[Thm.~4.1]{Neeman-Brown-96}, the functor $j_{!}$ admits a right adjoint $j^{*}:\mathbf{K}(\FPI(\class A))\rightarrow\mathbf{K}(\class C)$. The kernel of this right adjoint is $\ker(j^{*})=\{Y\,|\,\forall X\in\mathbf{K}(\class C),\,\, \Hom_{\mathbf{K}(\FPI(\class A))}(X,Y)=0\}$, which by \ref{lemma-pairs} (ii) is precisely the category $\mathbf{K}_{\mathrm{pac}}(\FPI(\class A))$.

Therefore, well known arguments (see for instance \cite[Remark 2.12]{ANm08}) imply that the composite $\mathbf{K}(\class C)\xrightarrow{j_{!}}\mathbf{K}(\FPI(\class A))\xrightarrow{\mathrm{can}}\mathbf{D}(\FPI(\class A))$
is an equivalence of triangulated categories and that the canonical map $\mathbf{K}(\FPI(\class A))\rightarrow\mathbf{D}(\FPI(\class A))$ is equivalent (up to natural isomorphism) with $j^{*}$.
\end{proof}

\begin{remark}
For any locally finitely presented Grothendieck category $\class A$, Krause in \cite[Example 7]{Krause-approx-adj} shows the existence of a left adjoint of the canonical map $\mathbf{K}(\FPI(\class A))\rightarrow\mathbf{D}(\FPI(\class A))$. In the proof of \ref{main thm} above, we obtain such a left adjoint after restricting ourselves to the case where $\class A$ is locally coherent, and we identify its essential image with $\mathbf{K}(\FPI(\class A)\cap\leftidx{^{\bot}}\FPI(\class A)).$
\end{remark}

Before closing this note, we mention that our theorem \ref{main thm} has an interpretation in the language of (Quillen) model categories. By the work of Hovey \cite{hovey} (resp. Gillespie \cite{Gil2011}) we know that certain cotorsion pairs on an abelian (resp.~exact) category $\class A$ correspond bijectively to the so-called abelian (resp.~exact) model structures on the category $\class A$. 
If $\class A$ is a locally coherent Grothendieck category, it is not hard to see that the cotorsion pair on the category $\Ch(\FPI(\class A))$ we obtained in \ref{lemma-pairs}, corresponds (via the aforementioned Hovey--Gillespie theory) to an exact model structure on the category $\Ch(\FPI(\class A))$ with Quillen homotopy category $\mathbf{D}(\FPI(\class A))$. The precise statement is as follows.

\begin{theorem}
\label{Model category}
Let $\class A$ be a locally coherent Grothendieck category and let $\Ch(\FPI(\class A))$ denote the category of chain complexes with components FP-injective objects. Then there exists an (exact) model structure on $\Ch(\FPI(\class A))$, where 
\begin{itemize}
\item[-] the cofibrant objects are the chain complexes in $\Ch(\FPI(\class A)\cap\leftidx{^{\bot}}\FPI(\class A))$.
\item[-] every chain complex in $\Ch(\FPI(\class A))$ is fibrant.
\item[-] the trivial objects are the pure acyclic complexes with FP-injective components.
%\item[-] the trivially fibrant objects are contractible complexes with FP-Injective components.
\end{itemize}
The homotopy category of this model structure is equivalent to $\mathbf{D}(\FPI(\class A))$.
\end{theorem}

\begin{remark}
\label{Stov}
Let $\class A$ be a locally coherent Grothendieck category. \v{S}{\fontencoding{T1}\selectfont \v{t}}ov\'i\v{c}ek in \cite[Thm.~6.12]{Stoviceck-on-purity} shows the existence of a model structure with Quillen homotopy category $\mathbf{D}(\FPI(\class A))$ and also proves an equivalence $\mathbf{D}(\FPI(\class A))\cong\mathbf{K}(\Inj(\class A))$. In \ref{Model category} we identify the cofibrant objects of this model structure with the category  $\Ch(\FPI(\class A)\cap\leftidx{^{\bot}}\FPI(\class A))$. Thus, combining \ref{Model category} with \cite[Thm.~6.12]{Stoviceck-on-purity} we obtain equivalences
\[\mathbf{K}(\FPI(\class A)\cap\leftidx{^{\bot}}\FPI(\class A))\cong\mathbf{D}(\FPI(\class A))\cong\mathbf{K}(\Inj(\class A)).\]
\end{remark}

\vspace*{3mm}

\section*{Acknowledgement}
The author would like to thank his PhD supervisors, Sergio Estrada from the University of Murcia and Henrik Holm from the University of Copenhagen.

\bibliographystyle{amsplain}
\bibliography{/Users/Tzo/Desktop/Latexxx/FPinjective.bib}

\end{document}